\documentclass[12pt, reqno]{amsart}
\usepackage{amsmath, amsthm, amscd, amsfonts, amssymb, graphicx, color}
\usepackage[bookmarksnumbered, colorlinks, plainpages]{hyperref}
\hypersetup{colorlinks=true,linkcolor=red, anchorcolor=green, citecolor=cyan, urlcolor=red, filecolor=magenta, pdftoolbar=true}
\input{mathrsfs.sty}

\newtheorem{theorem}{Theorem}[section]
\newtheorem{lemma}[theorem]{Lemma}

\newtheorem{cor}[theorem]{Corollary}
\theoremstyle{definition}

\theoremstyle{remark}
\newtheorem{remark}[theorem]{\bf{Remark}}
\numberwithin{equation}{section}
\begin{document}
	
 \title[Berezin number inequalities]{Development of  the Berezin number inequalities }
 
 \author[ P. Bhunia, A. Sen, K. Paul] { Pintu Bhunia, Anirban Sen, Kallol Paul}

 \address [Bhunia] {Department of Mathematics, Jadavpur University, Kolkata 700032, West     Bengal, India}
 \email{pintubhunia5206@gmail.com}
 \email{pbhunia.math.rs@jadavpuruniversity.in}

 \address [Sen] {Department of Mathematics, Jadavpur University, Kolkata 700032, West Bengal, India}
 \email{anirbansenfulia@gmail.com}

 \address[Paul] {Department of Mathematics, Jadavpur University, Kolkata 700032, West Bengal, India}
 \email{kalloldada@gmail.com}
 \email{kallol.paul@jadavpuruniversity.in}

 \thanks{First  author would like to thank UGC, Govt. of India for the financial support in the form of SRF. Second author would like to thank  CSIR, Govt. of India for the financial support in the form of JRF}
	
	
	\maketitle

\begin{abstract}
We present  new bounds for  the Berezin number inequalities which improve on the existing bounds. We also obtain bounds for the Berezin norm of operators as well as the sum of two operators.
\end{abstract}

%%% ----------------------------------------------------------------------
\maketitle
%%% ----------------------------------------------------------------------
%\tableofcontents

\section{\textbf{Introduction}}

\noindent 
Let $\mathbb{F}$ denote the field of scalars, real or complex, $\Omega $ be a non-empty set and $ \mathcal{F}(\Omega, \mathbb{F})$ be the set of functions from $\Omega $ to $ \mathbb{F}$. A set $\mathcal{H} \subset  \mathcal{F}(\Omega, \mathbb{F})$  is said to be a reproducing kernel Hilbert space (RKHS) on $\Omega$ if  $ \mathcal{H}$ is a Hilbert space and for every $ \lambda \in \Omega$, the linear evaluation functional $ E_{\lambda} : \mathcal{H} \to \mathbb{F}$, defined by $ E_{\lambda}(f) = f(\lambda),$ is bounded. Using Riesz representation theorem it is easy to see that for each $ \lambda \in \Omega$, there exists a unique vector $k_\lambda \in \mathcal{H}$, such that for every $ f \in \mathcal{H}$, $ f(\lambda) = E_\lambda(f) = \langle f, k_\lambda \rangle.$ The function $ k_\lambda$ is called the reproducing kernel for the element $\lambda$  and the set $\{ k_\lambda : \lambda \in \Omega\} $ is called the reproducing kernel of $\mathcal{H}.$  For $\lambda\in\Omega,$ let $\hat{k}_{\lambda} = k_\lambda/\|k_\lambda\|$ be a normalised reproducing kernel of $\mathcal{H}.$

\noindent 
Let  A be a bounded linear operator on $\mathcal{H}.$ The function $\widetilde{A}$ defined on $\Omega$ by $\widetilde{A}(\lambda)=\langle A\hat{k}_{\lambda},\hat{k}_{\lambda} \rangle$  is the Berezin symbol of $A$, see \cite{PR}.  The Berezin set and the Berezin number of the operator A are defined respectively by $$\textbf{Ber}(A):= \{\langle A\hat{k}_{\lambda},\hat{k}_{\lambda} \rangle: \lambda \in \Omega\} $$ and $$\textbf{ber}(A):=sup \{|\langle A\hat{k}_{\lambda},\hat{k}_{\lambda} \rangle|~:~ \lambda \in \Omega\}.$$
The numerical range and the numerical radius of $A$, to be denoted by $W(A) $ and $w(A)$ respectively,  are defined as 
\[ W(A) := \{ \langle Af,f \rangle : f \in \mathcal{H}, \|f\|=1\} ~\mbox{and} ~ w(A) := \sup_{\|f\|=1} | \langle Af,f \rangle|.\]
For some results about the numerical radius inequalities and their applications, we refer to see \cite{p1,p2,BPRIM,p3,p4,p5,EK}. 
It is well known that the numerical range is a convex set and the closure of the numerical range contains the spectrum of the operator $A$.  Clearly $ \textbf{Ber}(A) \subset W(A) $ and so $ \textbf{ber}(A) \leq w(A).$  Note that  $ \textbf{Ber}(A) $ is not necessarily convex.
The Berezin number of  $A$ satisfies the following: $ \textbf{ber}(\alpha A) = | \alpha |  \textbf{ber}(A) $ for all $ \alpha  \in \mathbb{F}$ and $ \textbf{ber}(A + B) \leq \textbf{ber}(A)  + \textbf{ber}(B)$  for bounded linear operators $A, B$ on $ \mathcal{H.}$  
It is well known (see  \cite[Prop. 6.2]{zhu} and \cite[Lemma 2]{karaev}) that any reproducing kernel Hilbert space of analytic
functions in the unit disc $\mathbb{D}$ (including Hardy and Bergman spaces) has the (Ber) property, i.e., the Berezin symbol of an operator determines the operator  uniquely in the sense that if $ \widetilde{A}(\lambda) = \widetilde{B}(\lambda) $ for all $ \lambda \in \Omega$ then $ A = B.$ 
The Berezin norm of A is defined by $$\|A\|_{ber}:= sup\{|\langle A\hat{k}_{\lambda},\hat{k}_{\mu} \rangle|: \lambda,\mu \in \Omega\},$$
where $ \hat{k}_{\lambda}$, $\hat{k}_{\mu}$ are normalized reproducing kernels for $ \lambda$, $\mu$, respectively.
Clearly from the definition, Berezin norm satisfy the following results: 
\begin{eqnarray*}
&&(i)\, \textbf{ber}(A)\leq \|A\|_{ber},\\
&&(ii)\, \|A\|_{ber}\leq \|A\|,\\
&&(iii)\, \|A^*\|_{ber}=\|A\|_{ber}.
\end{eqnarray*}
The Berezin symbol has been studied in details for Toeplitz and Hankel operators on Hardy and Bergman spaces. The Berezin symbol and Berezin number have been studied by many mathematicians over the years, a few of them are \cite{BLHY,hlb,ki,Trd}. In this paper, we study the Berezin number inequalities to develop upper and lower bounds for the Berezin number of bounded linear operators. The bounds obtained here a improve on the existing bounds. We also obtain bounds for the Berezin norm of operators as well as the sum of two operators.\\

\section{Berezin number inequalities}

\noindent
We begin this section with the following well known lemmas. First, note that $\mathcal{B}(\mathcal{H})$ stands for the set of all bounded linear operators on $\mathcal{H}$.

\begin{lemma}
	$($\cite[pp. 75-76]{halmos}$)$\label{lemma1}
	Let $A\in \mathcal{B}(\mathcal{H})$ and let $x\in \mathcal{H}$. Then 
	\[|\langle Ax,x\rangle|\leq \langle |A|x,x\rangle^{1/2}~~\langle |A^*|x,x\rangle^{1/2}.\]
\end{lemma}

\begin{lemma}
	$($\cite[p. 20]{simon}$)$\label{lemma2}
	Let $A\in \mathcal{B}(\mathcal{H})$ be positive and let $x\in \mathcal{H}$ with $\|x\|=1$. Then \[\langle Ax,x\rangle^r\leq \langle A^rx,x\rangle, ~~\forall~~r\geq 1.\]
\end{lemma}

\begin{lemma}$(${\cite[Th. 5]{K88}}$)$.\label{lem1}
	Let $A\in \mathcal{B}(\mathcal{H})$. Let $f$ and $g$ be nonnegative functions on $[0,\infty)$ which are continuous and satisfy the relation $f(t)g(t)=t$ for all $t\in [0,\infty).$ Then \[|\langle Ax,y\rangle|\leq  \|f(|A|)x\| \|g(|A^*|)y\|, \] for all $x,y \in \mathcal{H}. $
\end{lemma}

	%\begin{lemma}$($\cite[p. 20]{simon}$)$.\label{lem2}
	%Let $A\in \mathcal{B}(\mathcal{H})$ be positive, i.e., $A\geq 0.$ Then \[\langle Ax,x\rangle^r\leq \langle A^rx,x\rangle,\] for all $r\geq 1$ and for all $x \in \mathcal{H}$ with $\|x\|=1.$
%\end{lemma}

\begin{lemma}$($\cite{buzano}$)$\label{lem3}
	Let $x,y,e\in \mathcal{H}$ with $\|e\|=1.$ Then 
	\[|\langle x,e\rangle \langle e,y\rangle|\leq \frac{1}{2}\left(\|x\| \|y\|+|\langle x,y\rangle|\right).\]
\end{lemma}
	
\begin{lemma}$($\cite{v}$)$\label{lem4}
	For $i=1,2,\ldots,n$, let $a_i$ be a positive real number. Then 
	\[\left( \sum_{i=1}^na_i\right)^r \leq n^{r-1}\sum_{i=1}^na_i^r,\] for all $r\geq 1$.
\end{lemma}

We are now ready to present the first result of this section.

\begin{theorem}\label{theo1}
	Let $A\in\mathcal{B}(\mathcal{H})$. Then \[\textbf{ber}^{2r}(A)\leq \textbf{ber}\left(\alpha|A|^{2r}+(1-\alpha)|A^*|^{2r}\right),\]
	for all $r\geq 1$ and for all $\alpha\in [0,1].$
\end{theorem}

\begin{proof}
	Let $\hat{k}_{\lambda}$ be a normalised reproducing kernel of $\mathcal{H}.$ Then for all $\alpha\in [0,1]$, we have
	\begin{eqnarray*}
	|\langle A\hat{k}_{\lambda},\hat{k}_{\lambda}\rangle|& =& \alpha |\langle A\hat{k}_{\lambda},\hat{k}_{\lambda}\rangle|+(1-\alpha)|\langle A\hat{k}_{\lambda},\hat{k}_{\lambda}\rangle|\\
	& & \leq \alpha |\langle A\hat{k}_{\lambda},\hat{k}_{\lambda}\rangle|+(1-\alpha)\|A^*\hat{k}_{\lambda}\|.
	\end{eqnarray*}
	By convexity of  $f(t)=t^{2r}~(r \geq1)$, we get
	\begin{align*}
	    &|\langle A\hat{k}_{\lambda},\hat{k}_{\lambda}\rangle|^{2r}\\&\leq \alpha |\langle  A\hat{k}_{\lambda},\hat{k}_{\lambda}\rangle|^{2r}+(1-\alpha)\|A^*\hat{k}_{\lambda}\|^{2r}\\
		&\leq  \alpha \langle |A|\hat{k}_{\lambda},\hat{k}_{\lambda}\rangle^{r}~~\langle |A^*|\hat{k}_{\lambda},\hat{k}_{\lambda}\rangle^{r}+(1-\alpha)\langle |A^*|^2\hat{k}_{\lambda},\hat{k}_{\lambda}\rangle ^{r}, ~~\mbox{by Lemma \ref{lemma1}}\\
		&\leq  \alpha \langle |A|^{r}\hat{k}_{\lambda},\hat{k}_{\lambda}\rangle~~\langle |A^*|^{r}\hat{k}_{\lambda},\hat{k}_{\lambda}\rangle+(1-\alpha)\langle |A^*|^{2r}\hat{k}_{\lambda},\hat{k}_{\lambda}\rangle, ~~\mbox{by Lemma \ref{lemma2}}\\
		&\leq \frac{\alpha}{2} \left (\langle |A|^{r}\hat{k}_{\lambda},\hat{k}_{\lambda}\rangle^2 + \langle |A^*|^{r}\hat{k}_{\lambda},\hat{k}_{\lambda}\rangle^2 \right ) +(1-\alpha)\langle |A^*|^{2r}\hat{k}_{\lambda},\hat{k}_{\lambda}\rangle\\
		&\leq  \frac{\alpha}{2} \left (\langle |A|^{2r}\hat{k}_{\lambda},\hat{k}_{\lambda} \rangle + \langle |A^*|^{2r}\hat{k}_{\lambda},\hat{k}_{\lambda} \rangle \right ) +(1-\alpha) \langle |A^*|^{2r}\hat{k}_{\lambda},\hat{k}_{\lambda} \rangle\\
		&= \left \langle \left \{ \frac{\alpha}{2} \left ( |A|^{2r}+ |A^*|^{2r}\right )+(1-\alpha) |A^*|^{2r} \right\} \hat{k}_{\lambda},\hat{k}_{\lambda} \right \rangle\\ 
		&\leq  \textbf{ber}\left(\frac{\alpha}{2}|A|^{2r}+(1-\frac{\alpha}{2})|A^*|^{2r}\right).
	\end{align*}
	Taking supremum over all $\lambda\in\Omega$, we get
	\begin{eqnarray}\label{eqn1}
	\textbf{ber}^{2r}(A)\leq \textbf{ber} \left( \frac{\alpha}{2} |A|^{2r}+(1-\frac{\alpha}{2}) |A^*|^{2r}\right).
	\end{eqnarray}
	Replacing $A$ by $A^*$ in (\ref{eqn1}), we get
	\begin{eqnarray}\label{eqn2}
	\textbf{ber}^{2r}(A)\leq \textbf{ber}\left((1-\frac{\alpha}{2})  |A|^{2r}+\frac{\alpha}{2} |A^*|^{2r}\right).
	\end{eqnarray}
	Combining (\ref{eqn1}) and (\ref{eqn2}), we get
	\[\textbf{ber}^{2r}(A)\leq \textbf{ber}\left(\alpha|A|^{2r}+(1-\alpha)|A^*|^{2r}\right).\] 
	
\end{proof}

Considering $r=1$ in Theorem \ref{theo1} we get the following corollary.

\begin{cor}\label{cor11}
	Let $A\in\mathcal{B}(\mathcal{H})$. Then 
	\begin{eqnarray}\label{inq1}
		\textbf{ber}^{2}(A)&\leq& \min_{\alpha\in  [0,1]}\textbf{ber}\left(\alpha A^*A+(1-\alpha)AA^*\right)\\
		&\leq& \frac{1}{2}\textbf{ber}\left(A^*A+AA^*\right).
	\end{eqnarray}
	
\end{cor}

\begin{remark}
	Clearly the inequallity (\ref{inq1}) is better than the existing inequality
    $\textbf{ber}^{2}(A)\leq\frac{1}{2}\textbf{ber}(|A|^{2}+|A^*|^{2}),$ obtained in \cite[Cor. 3.5 (i)]{Trd}.
	
\end{remark}

Our next result is 
	
\begin{theorem}\label{theo3} 
	Let $A\in \mathcal{B}(\mathcal{H})$. Then for all $r\geq 1$ and for all $\alpha\in [0,1],$ 
	\begin{eqnarray}\label{eqn 3}
	\textbf{ber}^{2r}(A) \leq \frac{\alpha}{2}\textbf{ber}^r(A^2)+\textbf{ber}\left( \frac{\alpha}{4}|A|^{2r}+\left(1-\frac{3}{4}\alpha\right)|A^*|^{2r}\right),
	\end{eqnarray}
and 
	\begin{eqnarray}\label{eqn 4}
	\textbf{ber}^{2r}(A) \leq \frac{\alpha}{2}\textbf{ber}^r(A^2)+\textbf{ber}\left( \left(1-\frac{3}{4}\alpha\right)|A|^{2r}+\frac{\alpha}{4}|A^*|^{2r}\right).
	\end{eqnarray}
\end{theorem}

\begin{proof}
	Let $\hat{k}_{\lambda}$ be a normalised reproducing kernel of $\mathcal{H}.$ Then using Cauchy-Schwarz inequality, we get
	 \begin{eqnarray*}
	 	|\langle A\hat{k}_{\lambda},\hat{k}_{\lambda}\rangle|& = & \alpha |\langle A\hat{k}_{\lambda},\hat{k}_{\lambda}\rangle|+(1-\alpha)|\langle A\hat{k}_{\lambda},\hat{k}_{\lambda}\rangle|\\
	 	& & \leq \alpha |\langle A\hat{k}_{\lambda},\hat{k}_{\lambda}\rangle|+(1-\alpha)\|A^*\hat{k}_{\lambda}\|.
	 	\end{eqnarray*}
	By convexity of  $f(t)=t^{2r}~(r \geq1)$, we get
	\begin{align*}
		&|\langle A\hat{k}_{\lambda},\hat{k}_{\lambda}\rangle|^{2r}\\&\leq \alpha |\langle A\hat{k}_{\lambda},\hat{k}_{\lambda}\rangle|^{2r}+(1-\alpha)\|A^*\hat{k}_{\lambda}\|^{2r}\\ 
		&\leq\alpha |\langle A\hat{k}_{\lambda},\hat{k}_{\lambda}\rangle|^{2r}+(1-\alpha)\langle |A^*|^{2r}\hat{k}_{\lambda},\hat{k}_{\lambda}\rangle, ~~\mbox{by Lemma \ref{lemma2}}\\
		&\leq \frac{\alpha }{2} |\langle A^2\hat{k}_{\lambda},\hat{k}_{\lambda}\rangle|^{r}+\frac{\alpha }{4} \left \langle (|A|^{2r}+|A^*|^{2r})\hat{k}_{\lambda},\hat{k}_{\lambda} \right \rangle+(1-\alpha)\langle |A^*|^{2r}\hat{k}_{\lambda},\hat{k}_{\lambda}\rangle, \\
		&\,\,\,\,\,\,\,\,\,\,\,\,\,\,\,\,\,\,\,\,\,\,\,\,\,\,\,\,\,\,\,\,\,\,\,\,\,\,\,\,\,\,\,\,\,\,\,\,\,\,\,\,\,\,\,\,\,\,\,\,\,\,\,\,\,\,\,\,\,\,\,\,\,\,\,\,\,\,\,\,\,\,\,\,\,\,\,\,\,\,\,\,\,\,\,\,\,\,\,\,\,\,\,\,\,\,\,\,\,~~\mbox{by Lemma \ref{lem3}}\\
		&= \frac{\alpha }{2} |\langle A^2\hat{k}_{\lambda},\hat{k}_{\lambda}\rangle|^{r}+ \left \langle \left \{ \frac{\alpha}{4} \left ( |A|^{2r}+ |A^*|^{2r}\right )+(1-\alpha) |A^*|^{2r} \right\} \hat{k}_{\lambda},\hat{k}_{\lambda} \right \rangle\\
		&= \frac{\alpha }{2} |\langle A^2\hat{k}_{\lambda},\hat{k}_{\lambda}\rangle|^{r}+\left \langle \left \{  \frac{\alpha}{4}   |A|^{2r}+\left(1-\frac{3}{4}\alpha\right)  |A^*|^{2r}     \right\} \hat{k}_{\lambda},\hat{k}_{\lambda} \right \rangle\\
		&\leq \frac{\alpha}{2} \textbf{ber}^r(A^2) +\textbf{ber}\left( \frac{\alpha}{4}   |A|^{2r}+\left(1-\frac{3}{4}\alpha\right)  |A^*|^{2r}  \right).
	\end{align*}
	Taking supremum over all $\lambda\in\Omega$, we get (\ref{eqn 3}). Replacing $A$ by $A^*$ in  (\ref{eqn 3}), we get  (\ref{eqn 4}). This compleates our proof.
	
\end{proof}

Considering  $r=1$ in (\ref{eqn 3}) and (\ref{eqn 4}) we get the following corollary.

\begin{cor} \label{cor2}
	Let $A \in \mathcal{B}(\mathcal{H}).$ Then 
	\begin{eqnarray}\label{bound2comb}
	\textbf{ber}^{2}(A)\leq \min \left \{\beta_1, \beta_2  \right \}, 
	\end{eqnarray}
	\mbox{where}
	\[ \beta_1 = \min_{0\leq \alpha \leq 1} \left[\frac{\alpha}{2} \textbf{ber}^r(A^2) +\textbf{ber}\left( \frac{\alpha}{4}|A|^{2r}+\left(1-\frac{3}{4}\alpha\right)|A^*|^{2r} \right)\right]\] and
	\[\beta_2 = \min_{0\leq \alpha \leq 1}\left[ \frac{\alpha}{2} \textbf{ber}^r(A^2) +\textbf{ber}\left( \left(1-\frac{3}{4} \alpha\right)  |A|^{2r}+\frac{\alpha}{4} |A^*|^{2r} \right)\right].\]
\end{cor}

If we consider $\alpha=1$ in (\ref{eqn 3}) or (\ref{eqn 4}),  we get the following inequality.

\begin{cor}\label{theo3special}
	Let $A\in\mathcal{B}(\mathcal{H})$. Then 
	\[\textbf{ber}^{2r}(A)\leq \frac{1}{2}\textbf{ber}^{r}(A^2)+\frac{1}{4}\textbf{ber}(|A|^{2r}+|A^*|^{2r}),\]
	for all $r\geq 1$ and for all $\alpha\in [0,1].$	
\end{cor}

\begin{remark}
	If $A^2=0$ then Corollary \ref{theo3special} gives, $\textbf{ber}^{2r}(A)\leq\frac{1}{4}\textbf{ber}(|A|^{2r}+|A^*|^{2r})$	for all $r\geq 1.$ This is stronger than the existing inequality $\textbf{ber}^{2r}(A)\leq\frac{1}{2}\textbf{ber}(|A|^{2r}+|A^*|^{2r})$, obtained in \cite[Cor. 3.5 (i)]{Trd}.
	
\end{remark}

\begin{theorem}\label{theo4}
	Let $A\in \mathcal{B}(\mathcal{H})$. Then for all $r\geq 1$ and for all $\alpha\in [0,1],$  we have
	\begin{eqnarray}\label{eqn 5}
	\textbf{ber}^{2r}(A) & \leq & \textbf{ber}\left( \alpha   \left(\frac{|A|+|A^*|}{2}\right)^{2r}+\left(1-\alpha\right)  |A^*|^{2r}  \right)
	\end{eqnarray}
and 
	\begin{eqnarray}\label{eqn 6}
	\textbf{ber}^{2r}(A) & \leq & \textbf{ber}\left( \alpha   \left(\frac{|A|+|A^*|}{2}\right)^{2r}+\left(1-\alpha\right)  |A|^{2r}  \right).
	\end{eqnarray}
\end{theorem}

\begin{proof}
	Let $\hat{k}_{\lambda}$ be a normalised reproducing kernel of $\mathcal{H}.$ Then using Cauchy-Schwarz inequality, we get 
	\begin{eqnarray*}
		|\langle A\hat{k}_{\lambda},\hat{k}_{\lambda}\rangle|& = & \alpha |\langle A\hat{k}_{\lambda},\hat{k}_{\lambda}\rangle|+(1-\alpha)|\langle A\hat{k}_{\lambda},\hat{k}_{\lambda}\rangle|\\
		& & \leq \alpha |\langle A\hat{k}_{\lambda},\hat{k}_{\lambda}\rangle|+(1-\alpha)\|A^*\hat{k}_{\lambda}\|
		\end{eqnarray*}
	By convexity of  $f(t)=t^{2r}~(r \geq1)$, we get
	\begin{align*}
		&|\langle A\hat{k}_{\lambda},\hat{k}_{\lambda}\rangle|^{2r}\\&\leq \alpha |\langle A\hat{k}_{\lambda},\hat{k}_{\lambda}\rangle|^{2r}+(1-\alpha)\|A^*x\|^{2r}\\ 
		&\leq \alpha |\langle A\hat{k}_{\lambda},\hat{k}_{\lambda}\rangle|^{2r}+(1-\alpha)\langle |A^*|^{2r}\hat{k}_{\lambda},\hat{k}_{\lambda}\rangle, ~~\mbox{by Lemma \ref{lemma2}}\\
		&\leq \alpha \left(\langle |A|\hat{k}_{\lambda},\hat{k}_{\lambda}\rangle^{1/2}\langle |A^*|\hat{k}_{\lambda},\hat{k}_{\lambda}\rangle^{1/2}\right)^{2r}+(1-\alpha)\langle |A^*|^{2r}\hat{k}_{\lambda},\hat{k}_{\lambda}\rangle,\\ &\,\,\,\,\,\,\,\,\,\,\,\,\,\,\,\,\,\,\,\,\,\,\,\,\,\,\,\,\,\,\,\,\,\,\,\,\,\,\,\,\,\,\,\,\,\,\,\,\,\,\,\,\,\,\,\,\,\,\,\,\,\,\,\,\,\,\,\,\,\,\,\,\,\,\,\,\,\,\,\,\,\,\,\,\,\,\,\,\,\,\,\,\,\,\,\,\,\,\,\,\,\,\,\,\,\,\,\,\,\,~~\mbox{by Lemma \ref{lemma1}}\\
		&\leq \alpha \left(\frac{\langle |A|\hat{k}_{\lambda},\hat{k}_{\lambda}\rangle+\langle |A^*|\hat{k}_{\lambda},\hat{k}_{\lambda}\rangle}{2} \right)^{2r}+(1-\alpha)\langle |A^*|^{2r}\hat{k}_{\lambda},\hat{k}_{\lambda}\rangle \\
		&= \alpha \left(\frac{\langle (|A|+ |A^*|)\hat{k}_{\lambda},\hat{k}_{\lambda}\rangle}{2} \right)^{2r}+(1-\alpha)\langle |A^*|^{2r}\hat{k}_{\lambda},\hat{k}_{\lambda}\rangle \\
		&\leq \alpha \left \langle \left(\frac{|A|+ |A^*|}{2}\right)^{2r}\hat{k}_{\lambda},\hat{k}_{\lambda} \right \rangle +(1-\alpha)\langle |A^*|^{2r}\hat{k}_{\lambda},\hat{k}_{\lambda}\rangle, ~~\mbox{by Lemma \ref{lemma2}}\\
		&= \left \langle \left\{   \alpha   \left(\frac{|A|+|A^*|}{2}\right)^{2r}+\left(1-\alpha\right)  |A^*|^{2r}   \right\}\hat{k}_{\lambda}, \hat{k}_{\lambda} \right \rangle\\
		& \leq  \textbf{ber}\left(\alpha   \left(\frac{|A|+|A^*|}{2}\right)^{2r}+\left(1-\alpha\right)  |A^*|^{2r}  \right).
	\end{align*}
	Taking supremum over all $\lambda\in\Omega$, we get  (\ref{eqn 5}). Replacing $A$ by $A^*$ in  (\ref{eqn 5}), we get  (\ref{eqn 6}).
	\end{proof}

Considering $r=1$ in (\ref{eqn 5}) and (\ref{eqn 6}) we get the following corollary.

\begin{cor}\label{cor 3}
	Let $A\in \mathcal{B}(\mathcal{H})$. Then  
	\begin{eqnarray}\label{bound3comb}
	\textbf{ber}^{2}(A)&\leq &\min \{ \gamma_1, \gamma_2\},
	\end{eqnarray}
	$$ \gamma_1 = \min_{0\leq \alpha \leq 1} \textbf{ber}\left( \alpha   \left(\frac{|A|+|A^*|}{2}\right)^{2}+\left(1-\alpha\right)  |A^*|^{2}  \right) $$ 
	and
	$$ \gamma_2 = \min_{0\leq \alpha \leq 1} \textbf{ber}\left( \alpha   \left(\frac{|A|+|A^*|}{2}\right)^{2}+\left(1-\alpha\right)  |A|^{2}  \right).$$
\end{cor}

 \begin{theorem}\label{theo5}
	Let $A\in \mathcal{B}(\mathcal{H})$. Then for all $r\geq1, $
	\[\textbf{ber}^{2r}(A)\leq \frac{1}{4}\textbf{ber} \left ( |A|^{2r}+|A^*|^{2r} \right)+\frac{1}{2}\textbf{ber} \left(|A|^r|A^*|^r\right),  \]
\end{theorem}

\begin{proof}
	Let $\hat{k}_{\lambda}$ be a normalised reproducing kernel of $\mathcal{H}.$ Then  if follows from Lemma \ref{lemma1} and Lemma \ref{lemma2}  that  \[|\langle A\hat{k}_{\lambda},\hat{k}_{\lambda}\rangle |^{2r} \leq \langle |A|^r\hat{k}_{\lambda},\hat{k}_{\lambda}\rangle \langle |A^*|^r\hat{k}_{\lambda},\hat{k}_{\lambda}\rangle=\langle |A^*|^r\hat{k}_{\lambda}, \hat{k}_{\lambda}\rangle\langle\hat{k}_{\lambda},|A|^r\hat{k}_{\lambda}\rangle.\]
	From Lemma \ref{lem3} we have,
	\begin{eqnarray*}
		\langle |A^*|^r\hat{k}_{\lambda},\hat{k}_{\lambda}\rangle \langle \hat{k}_{\lambda},|A|^r\hat{k}_{\lambda}\rangle
		&\leq& \frac{1}{2}\left\| |A|^r\hat{k}_{\lambda}\right\| \left\| |A^*|^r\hat{k}_{\lambda}\right\| +\frac{1}{2} \left | \langle |A^*|^r\hat{k}_{\lambda},|A|^r\hat{k}_{\lambda}\rangle \right |.
	\end{eqnarray*}
	\begin{align*}
		\mbox{Hence,}&\ |\langle A\hat{k}_{\lambda},\hat{k}_{\lambda}\rangle |^{2r}\\&\leq \frac{1}{4}\left (\left\| |A|^r\hat{k}_{\lambda}\right\|^2+ \left\| |A^*|^r\hat{k}_{\lambda}\right\|^2 \right) +\frac{1}{2} \left | \langle |A|^r|A^*|^r\hat{k}_{\lambda},\hat{k}_{\lambda}\rangle \right |\\
		&= \frac{1}{4}\left (\left \langle  |A|^{2r}\hat{k}_{\lambda},\hat{k}_{\lambda}\right \rangle + \left \langle  |A^*|^{2r}\hat{k}_{\lambda},\hat{k}_{\lambda}\right \rangle  \right) +\frac{1}{2} \left | \langle |A|^r|A^*|^r\hat{k}_{\lambda},\hat{k}_{\lambda}\rangle \right |\\
		&= \frac{1}{4} \left \langle \left ( |A|^{2r}+  |A^*|^{2r} \right) \hat{k}_{\lambda},\hat{k}_{\lambda}\right \rangle   +\frac{1}{2} \left | \langle |A|^r|A^*|^r\hat{k}_{\lambda},\hat{k}_{\lambda}\rangle \right |\\
		&\leq \frac{1}{4}\textbf{ber} \left ( |A|^{2r}+|A^*|^{2r} \right)+\frac{1}{2}\textbf{ber} \left(|A|^r|A^*|^r\right)
	\end{align*}
	Taking supremum over all $\lambda\in\Omega$, we get
	\begin{eqnarray*}
    \textbf{ber}^{2r}(A)\leq \frac{1}{4}\textbf{ber} \left ( |A|^{2r}+|A^*|^{2r} \right)+\frac{1}{2}\textbf{ber} \left(|A|^r|A^*|^r\right).
	\end{eqnarray*}
	
\end{proof}

Considering $r=1$ in Theorem \ref{theo5} we get the following corollary.

\begin{cor}
	Let $A\in\mathcal{B}(\mathcal{H})$. Then 
	\begin{eqnarray*}
		\textbf{ber}^{2}(A)\leq \frac{1}{4}\textbf{ber} \left ( |A|^{2}+|A^*|^{2} \right)+\frac{1}{2}\textbf{ber} \left(|A||A^*|\right). 
	\end{eqnarray*}
	
\end{cor}

Next we generalize Theorem \ref{theo5} in the following form.

\begin{theorem}\label{theo5gen}
	Let $A_i \in \mathcal{B}(\mathcal{H}), \ i=1,2, \ldots,n$.  Then for all $r\geq 1,$
	\begin{eqnarray*}
	\textbf{ber}^{2r}\left( \sum_{i=1}^{n}A_i \right)  
	& \leq & \frac{n^{2r-1}}{4}\textbf{ber} \left ( \sum_{i=1}^{n}  \left( |A_i|^{2r}+ |A_i^*|^{2r} \right)      \right)         +\frac{n^{2r-1}}{2} \left( \sum_{i=1}^{n}\textbf{ber} \left( |A_i|^{r} |A_i^*|^r \right)      \right).
    \end{eqnarray*}
\end{theorem}

\begin{proof}
	Let $\hat{k}_{\lambda}$ be a normalised reproducing kernel of $\mathcal{H}.$ From  Lemma \ref{lem4} we get, 
	\begin{eqnarray*}
		\left |\left \langle \left( \sum_{i=1}^{n}A_i \right)\hat{k}_{\lambda},\hat{k}_{\lambda} \right \rangle \right |^{2r} = \left |\sum_{i=1}^{n} \left \langle  A_i \hat{k}_{\lambda},\hat{k}_{\lambda} \right \rangle \right |^{2r}
		&\leq&  \left (\sum_{i=1}^{n} |\left \langle  A_i \hat{k}_{\lambda},\hat{k}_{\lambda} \right \rangle| \right )^{2r}\\
		&\leq&  n^{2r-1} \left (\sum_{i=1}^{n} |\left \langle  A_i \hat{k}_{\lambda},\hat{k}_{\lambda} \right \rangle|^{2r} \right ).
	\end{eqnarray*}
	 %\begin{align*}
		%&\left |\left \langle \left( \sum_{i=1}^{n}A_i \right)\hat{k}_{\lambda},\hat{k}_{\lambda} \right \rangle \right |^{2r}\\& = \left |\sum_{i=1}^{n} \left \langle  A_i \hat{k}_{\lambda},\hat{k}_{\lambda} \right \rangle \right |^{2r}\\
		%&\leq  \left (\sum_{i=1}^{n} |\left \langle  A_i \hat{k}_{\lambda},\hat{k}_{\lambda} \right \rangle| \right )^{2r}\\
		%&\leq  n^{2r-1} \left (\sum_{i=1}^{n} |\left \langle  A_i \hat{k}_{\lambda},\hat{k}_{\lambda} \right \rangle|^{2r} \right ).
    %\end{align*}
	Proceeding similarly  as in the proof of Theorem \ref{theo5} we get the required inequality.
\end{proof}

\begin{theorem}\label{theo6}
	Let $A_i,B_i, X_i\in \mathcal{B}(\mathcal{H}), i=1,2, \ldots,n$. Let $f$ and $g$ be two nonnegative functions on $[0,\infty)$ which are continuous and satisfy the relation $f(t)g(t)=t$ for all $t\in [0,\infty).$ Then for $\nu\in [0,1]$ and for all $r\geq 2 \ max\{\nu,1-\nu\},$ 
\begin{eqnarray*}
    & & \textbf{ber}^r\left( \sum_{i=1}^{n}A_i^*X_iB_i \right) \leq \\  
    && 
	\sqrt{2}n^{r-1}\textbf{ber}\left(  \sum_{i=1}^{n}\left(\nu \left[B_i^*f^2(|X_i|)B_i\right]^{\frac{r}{2\nu}}+{\rm i}(1-\nu)\left[A_i^*g^2(|X_i^*|)A_i\right]^{\frac{r}{2(1-\nu)}}\right) \right).
\end{eqnarray*}
\end{theorem}

\begin{proof}
	Let $\hat{k}_{\lambda}$ be a normalised reproducing kernel of $\mathcal{H}.$ Then we have,
	\begin{eqnarray*}
		&&\left |\left \langle \left( \sum_{i=1}^{n}A_i^*X_iB_i \right)\hat{k}_{\lambda},\hat{k}_{\lambda} \right \rangle \right |^r \\
		&=&\left |\sum_{i=1}^{n} \left \langle  A_i^*X_iB_i \hat{k}_{\lambda},\hat{k}_{\lambda} \right \rangle \right |^r\\
		&\leq& \left (\sum_{i=1}^{n} |\left \langle  A_i^*X_iB_i \hat{k}_{\lambda},\hat{k}_{\lambda} \right \rangle| \right )^r\\
		&\leq &n^{r-1} \left (\sum_{i=1}^{n} |\left \langle  A_i^*X_iB_i \hat{k}_{\lambda},\hat{k}_{\lambda} \right \rangle|^r \right ), ~~\mbox{by Lemma \ref{lem4}}\\
		&= &n^{r-1} \left (\sum_{i=1}^{n} |\left \langle  X_iB_i \hat{k}_{\lambda},A_i\hat{k}_{\lambda} \right \rangle|^r \right )\\
		& \leq &n^{r-1} \left (\sum_{i=1}^{n} \left \Vert  f(|X_i|)B_i \hat{k}_{\lambda} \right \Vert^r \left \Vert g(|X_i^*|)A_i \hat{k}_{\lambda} \right \Vert ^r\right ), ~~\mbox{by Lemma \ref{lem1}}\\
		&= &n^{r-1} \left (\sum_{i=1}^{n} \left \langle  f^2(|X_i|)B_i \hat{k}_{\lambda},B_i\hat{k}_{\lambda} \right \rangle^{\frac{r}{2}}  \left \langle  g^2(|X_i^*|)A_i \hat{k}_{\lambda},A_i\hat{k}_{\lambda} \right \rangle^{\frac{r}{2}}\right )\\
		&=& n^{r-1} \left (\sum_{i=1}^{n} \left \langle  B_i^*f^2(|X_i|)B_i \hat{k}_{\lambda},\hat{k}_{\lambda} \right \rangle^{\frac{r}{2}}  \left \langle  A_i^*g^2(|X_i^*|)A_i \hat{k}_{\lambda},\hat{k}_{\lambda} \right \rangle^{\frac{r}{2}}\right )\\
		&\leq &n^{r-1} \left (\sum_{i=1}^{n} \left \langle  \left[B_i^*f^2(|X_i|)B_i\right]^{\frac{r}{2\nu}} \hat{k}_{\lambda},\hat{k}_{\lambda} \right \rangle^{\nu}  \left \langle  \left[A_i^*g^2(|X_i^*|)A_i\right]^{\frac{r}{2(1-\nu)}} \hat{k}_{\lambda},\hat{k}_{\lambda} \right \rangle^{1-\nu} \right),\\ &&\,\,\,\,\,\,\,\,\,\,\,\,\,\,\,\,\,\,\,\,\,\,\,\,\,\,\,\,\,\,\,\,\,\,\,\,\,\,\,\,\,\,\,\,\,\,\,\,\,\,\,\,\,\,\,\,\,\,\,\,\,\,\,\,\,\,\,\,\,\,\,\,\,\,\,\,\,\,\,\,\,\,\,\,\,\,\,\,\,\,\,\,\,\,\,\,\,\,\,\,\,\,\,\,\,\,\,\,\,\,\,\,\,\,\,\,\,\,\,\,\,\,\,\,\,\,\,\,\,\,\,\,\,\,\,\,\,\,\,\,\,\,\,\,\,\,\,\,\,\,~~\mbox{by Lemma \ref{lemma2} }\\
		&\leq &n^{r-1} \Big (\sum_{i=1}^{n} \big(\nu \left \langle  \left[B_i^*f^2(|X_i|)B_i\right]^{\frac{r}{2\nu}} \hat{k}_{\lambda},\hat{k}_{\lambda} \right \rangle \\ &+&  (1-\nu) \left \langle  \left[A_i^*g^2(|X_i^*|)A_i\right]^{\frac{r}{2(1-\nu)}} \hat{k}_{\lambda},\hat{k}_{\lambda} \right \rangle  \big) \Big ), a^\nu b^{1-\nu}\leq \nu a +(1-\nu) b, a,b\geq 0\\
		&\leq &\sqrt{2}n^{r-1} \Big (\Big |\nu \sum_{i=1}^{n}  \left \langle  \left[B_i^*f^2(|X_i|)B_i\right]^{\frac{r}{2\nu}} \hat{k}_{\lambda},\hat{k}_{\lambda} \right \rangle   \\
	&+& { \rm i}(1-\nu) \sum_{i=1}^{n} \left \langle  \left[A_i^*g^2(|X_i^*|)A_i\right]^{\frac{r}{2(1-\nu)}} \hat{k}_{\lambda},\hat{k}_{\lambda} \right \rangle \Big| \Big),   |a+b|\leq \sqrt{2}|a+{\rm i }b|, a,b\in \mathbb{R} 
	\end{eqnarray*}
	\begin{eqnarray*}
		&=& \sqrt{2}n^{r-1}  \Big |\Big \langle  \Big( \sum_{i=1}^{n}\Big(\nu \left[B_i^*f^2(|X_i|)B_i\right]^{\frac{r}{2\nu}}\\
		&& +{\rm i}(1-\nu)\left[A_i^*g^2(|X_i^*|)A_i\right]^{\frac{r}{2(1-\nu)}}\Big)\Big) \hat{k}_{\lambda},\hat{k}_{\lambda} \Big \rangle \Big |\\
		&\leq &\sqrt{2}n^{r-1}\textbf{ber}\left( \sum_{i=1}^{n}\left(\nu \left[B_i^*f^2(|X_i|)B_i\right]^{\frac{r}{2\nu}}+{\rm i}(1-\nu)\left[A_i^*g^2(|X_i^*|)A_i\right]^{\frac{r}{2(1-\nu)}}\right)\right).
	\end{eqnarray*}
	Taking supremum over all $\lambda\in\Omega$, we get
    \begin{align*}
         &\textbf{ber}^r\left( \sum_{i=1}^{n}A_i^*X_iB_i \right)\\ & \leq 
         \sqrt{2}n^{r-1}\textbf{ber}\left(  \sum_{i=1}^{n}\left(\nu \left[B_i^*f^2(|X_i|)B_i\right]^{\frac{r}{2\nu}}+{\rm i}(1-\nu)\left[A_i^*g^2(|X_i^*|)A_i\right]^{\frac{r}{2(1-\nu)}}\right) \right).
    \end{align*}
\end{proof}

The following corollaries  now follow easily from Theorem \ref{theo6}.
\begin{cor}\label{theo6cor1}
	Let $A_i,B_i, X_i\in \mathcal{B}(\mathcal{H}), i=1,2, \ldots,n$.  Let $f$ and $g$ be two nonnegative functions on $[0,\infty)$ which are continuous and satisfy the relation $f(t)g(t)=t$ for all $t\in [0,\infty)$. Then for all $r\geq 1$ and $0\leq \alpha \leq 1$, we have
	\begin{eqnarray*}
	 &&(i)\,\textbf{ber}^r\left( \sum_{i=1}^{n}A_i^*X_iB_i \right)\\
	&& \,\,\,\,\,\,\,\,\,\,\,\, \leq \frac{n^{r-1}}{\sqrt{2}}\textbf{ber}\left(  \sum_{i=1}^{n}\left( \left[B_i^*f^2(|X_i|)B_i\right]^r+{\rm i}\left[A_i^*g^2(|X_i^*|)A_i\right]^r\right)      \right),  \\
	 &&(ii)\,\textbf{ber}^r\left( \sum_{i=1}^{n}X_i \right) \leq \frac{n^{r-1}}{\sqrt{2}}\textbf{ber}\left(  \sum_{i=1}^{n}\left( f^{2r}(|X_i|) +{\rm i} ~g^{2r}(|X_i^*|)  \right)      \right),  \\
 &&(iii)\, \textbf{ber}^r\left( \sum_{i=1}^{n}A_i^*X_iB_i \right) \\
&&\,\,\,\,\,\,\,\,\,\,\,\,  \leq \frac{n^{r-1}}{\sqrt{2}}\textbf{ber}\left(  \sum_{i=1}^{n}\left( \left[B_i^*|X_i|^{2\alpha}B_i\right]^r+{\rm i}\left[A_i^*|X_i^*|^{2(1-\alpha)}A_i\right]^r\right)      \right). 
\end{eqnarray*}
\end{cor}

\begin{cor}\label{theo6cor20}
	Let $A, B, X\in \mathcal{B}(\mathcal{H})$.  Then  for all $r\geq 1,$
	\begin{eqnarray*}
		& & (i)~\textbf{ber}\left(A^*XB \right) \leq \frac{1}{\sqrt{2}}\textbf{ber}\left(  \left( \left[B^*|X|^{2\alpha}B\right]+{\rm i}\left[A^*|X^*|^{2(1-\alpha)}A\right]\right)\right).\\
	 & & (ii)~\textbf{ber}\left(A^*XB \right) \leq \frac{1}{\sqrt{2}}\textbf{ber}\left(  \left( \left[B^*|X|B\right]+{\rm i}\left[A^*|X^*|A\right]\right)\right).
 	 \end{eqnarray*}
 \end{cor}
\begin{cor}\label{theo6cor4}
	Let $ A, B\in \mathcal{B}(\mathcal{H})$. Then 
	\begin{eqnarray*}
	& & (i)~	\textbf{ber}(A) \leq \frac{1}{\sqrt{2}} \textbf{ber}(|A|+{\rm i}|A^*|).\\
	& & (ii)~\textbf{ber}^{r}(A^*B) \leq \frac{1}{\sqrt{2}} \textbf{ber}\left (|B|^{2r}+{\rm i}|A|^{2r}\right), ~\mbox{for all } r \geq 1.
\end{eqnarray*}
\end{cor}

\begin{remark}
It is easy to observe that $\textbf{ber}^2(|A|+{\rm i}|A^*|)\leq \textbf{ber}(|A|^2+|A^*|^2).$ Thus, Corollary \ref{theo6cor4}(i) refines the existing bound \cite[Cor. 3.5 (i)]{Trd} $\textbf{ber}^2(A) \leq \frac{1}{{2}} \textbf{ber}(|A|^2+|A^*|^2).$

\end{remark}

\section{Berezin norm of operators}

\noindent We begin this section  with the following  two notations:  Let $\textit{Re}(A):=\frac{A+A^*}{2}$ and $\textit{Im}(A):=\frac{A-A^*}{2{\rm i}}$ be the real and imaginary part of a bounded linear operator $A $ on the Hilbert space $\mathcal{H}$ and $A=\textit{Re}(A)+{\rm i} \textit{Im}(A)$ be the Cartesian decomposition of $A$. Also, we define $ c(A):= inf\{|\langle A\hat{k}_{\lambda},\hat{k}_{\lambda} \rangle|: \lambda \in \Omega\}.$ 
First we obtain the following lower bound.

 \begin{theorem}
	\label{theo7}
	Let $A\in\mathcal{B}(\mathcal{H}).$ Then
	\begin{eqnarray*}
        \textbf{ber}^2(\textit{Re}(A))+\textbf{ber}^2(\textit{Im}(A))
		+ c^2(\textit{Re}(A))+c^2(\textit{Im}(A))
		&\leq& 2\textbf{ber}^2(A)\\
		&\leq& 2 \|A\|^2_{ber}.
	\end{eqnarray*}

\end{theorem}

 \begin{proof}
    We only prove the first inequality. Let $\hat{k}_{\lambda}$ be a normalised reproducing kernel of $\mathcal{H}.$ Then from the Cartesian decomposition of $A$ we get,
	\begin{eqnarray}\label{eqn10}
	|\langle B\hat{k}_{\lambda},\hat{k}_{\lambda}\rangle|^2+|\langle C\hat{k}_{\lambda},\hat{k}_{\lambda}\rangle|^2=|\langle A\hat{k}_{\lambda},\hat{k}_{\lambda}\rangle|^2.
	\end{eqnarray}
	From (\ref{eqn10}), we get the following two inequalities
	\begin{eqnarray}\label{eqn20}
	c^2(\textit{Re}(A))+\textbf{ber}^2(\textit{Im}(A))\leq \textbf{ber}^2(A)
	\end{eqnarray}
	and 
	\begin{eqnarray}\label{eqn30}
	c^2(\textit{Im}(A))+\textbf{ber}^2(\textit{Re}(A))\leq \textbf{ber}^2(A).
	\end{eqnarray}
		It follows from the inequalities (\ref{eqn20}) and (\ref{eqn30}) that 
	\begin{eqnarray*}
		c^2(\textit{Re}(A))+c^2(\textit{Im}(A))+\textbf{ber}^2(\textit{Re}(A))+\textbf{ber}^2(\textit{Im}(A)) \leq 2 \textbf{ber}^2(A).
	\end{eqnarray*}
	This completes the proof.
\end{proof}

Next result is

\begin{theorem} \label{theo8}
	Let $A, B\in \mathcal{B}(\mathcal{H}).$ Then
	\begin{align*}
		\|A+B\|^2_{ber} \leq  \|A\|^2_{ber}+\|B\|^2_{ber}+ \textbf{ber}^\frac{1}{2}(A^*A) \textbf{ber}^\frac{1}{2}(B^*B) +\textbf{ber}(A^*B).
	\end{align*}
\end{theorem}

\begin{proof}
	Let $\hat{k}_{\lambda}$ and $\hat{k}_{\mu}$ be two normalised reproducing kernel of $\mathcal{H}.$ Then we get,
	\begin{eqnarray*}
		& &|\langle (A+B)\hat{k}_{\lambda},\hat{k}_{\mu} \rangle |^2\\ 
		& \leq& (|\langle A\hat{k}_{\lambda},\hat{k}_{\mu} \rangle |+|\langle B\hat{k}_{\lambda},\hat{k}_{\mu} \rangle |)^2\\
		&= &|\langle A\hat{k}_{\lambda},\hat{k}_{\mu} \rangle |^2+|\langle B\hat{k}_{\lambda},\hat{k}_{\mu} \rangle |^2+2|\langle A\hat{k}_{\lambda},\hat{k}_{\mu} \rangle \langle B\hat{k}_{\lambda},\hat{k}_{\mu} \rangle |\\
		&= &|\langle A\hat{k}_{\lambda},\hat{k}_{\mu} \rangle |^2+|\langle B\hat{k}_{\lambda},\hat{k}_{\mu} \rangle |^2+2|\langle A\hat{k}_{\lambda},\hat{k}_{\mu} \rangle \langle \hat{k}_{\mu},B\hat{k}_{\lambda} \rangle |\\
		&\leq &|\langle A\hat{k}_{\lambda},\hat{k}_{\mu} \rangle |^2+|\langle B\hat{k}_{\lambda},\hat{k}_{\mu} \rangle |^2+\|A\hat{k}_{\lambda}\| \|B\hat{k}_{\lambda}\|+ |\langle A\hat{k}_{\lambda},B\hat{k}_{\lambda} \rangle |,\\
		&& \,\,\,\,\,\,\,\,\,\,\,\,\,\,\,\,\,\,\,\,\,\,\,\,\,\,\,\,\,\,\,\,\,\,\,\,\,\,\,\,\,\,\,\,\,\,\,\,\,\,\,\,\,\,\,\,\,\,\,\,\,\,\,\,\,\,\,\,\,\,\,\,\,\,\,\,\,\,\,\,\,\,\,\,\,\,\,\,\,\,\,\,\,\,\,\,\,\,\,\,\,\,\,\,\,\,\,\,\,\,\,\,\,\,\,\,\,\,\,\,\,\,\,\,\,\,\,\,\,\,\,\,~~\mbox{By Lemma \ref{lem3}}\\
		&\leq &|\langle A\hat{k}_{\lambda},\hat{k}_{\mu} \rangle |^2+|\langle B\hat{k}_{\lambda},\hat{k}_{\mu} \rangle |^2+|\langle A^*A\hat{k}_{\lambda},\hat{k}_{\lambda}\rangle |^\frac{1}{2}|\langle B^*B\hat{k}_{\lambda},\hat{k}_{\lambda}\rangle |^\frac{1}{2} \\
		&& \,\,\,\,\,\,\,\,\,\,\,\,\,\,\,\,\,\,\,\,\,\,\,\,\,\,\,\,\,\,\,\,\,\,\,\,\,\,\,\,\,\,\,\,\,\,\,\,\,\,\,\,\,\,\,\,\,\,\,\,\,\,\,\,\,\,\,\,\,\,\,\,\,\,\,\,\,\,\,\,\,\,\,\,\,\,\,\,\,\,\,\,\,\,\,\,\,\,\,\,\,\,\,\,\,\,\,\,\,\,\,\,\,\,\,\,\,\,\,\,\,\,\,\,\,\,\,\,\,\,\,\,+ |\langle A^*B\hat{k}_{\lambda},\hat{k}_{\lambda} \rangle |\\
		&\leq&  \|A\|^2_{ber}+\|B\|^2_{ber}+\textbf{ber}^\frac{1}{2}(A^*A) \textbf{ber}^\frac{1}{2}(B^*B)+ \textbf{ber}(A^*B).
	\end{eqnarray*}
Taking supremum over all $\lambda,\mu\in\Omega$, we get,
\begin{eqnarray*}
\|A+B\|^2_{ber} \leq \|A\|^2_{ber}+\|B\|^2_{ber}+\textbf{ber}^\frac{1}{2}(A^*A) \textbf{ber}^\frac{1}{2}(B^*B)+\textbf{ber}(A^*B).
\end{eqnarray*}
   
\end{proof}

Finally we obtain the following berezin norm inequality.

\begin{theorem} \label{theo9}
	Let $A, B\in \mathcal{B}(\mathcal{H}).$ Then the following inequalities hold:
	\begin{eqnarray*}\label{eqn16}
	\|A+B\|^2_{ber} \leq \|A\|^2_{ber}+\|B\|^2_{ber}+\frac{1}{2}\textbf{ber}(A^*A+B^*B)+\textbf{ber}(A^*B)
	\end{eqnarray*}
	and
	\begin{eqnarray*}\label{eqn17}
	\|A+B\|^2_{ber} \leq \|A\|^2_{ber}+\|B\|^2_{ber}+\frac{1}{2}\textbf{ber}(AA^*+BB^*)+\textbf{ber}(AB^*).
	\end{eqnarray*}
\end{theorem}

\begin{proof}
	Let $\hat{k}_{\lambda}$ and $\hat{k}_{\mu}$ be two normalised reproducing kernel of $\mathcal{H}.$ Then we get,
	\begin{align*}
		&|\langle (A+B)\hat{k}_{\lambda},\hat{k}_{\mu} \rangle |^2 \\ &\leq  (|\langle A\hat{k}_{\lambda},\hat{k}_{\mu} \rangle |+|\langle B\hat{k}_{\lambda},\hat{k}_{\mu} \rangle |)^2\\
		&= |\langle A\hat{k}_{\lambda},\hat{k}_{\mu} \rangle |^2+|\langle B\hat{k}_{\lambda},\hat{k}_{\mu} \rangle |^2+2|\langle A\hat{k}_{\lambda},\hat{k}_{\mu} \rangle \langle B\hat{k}_{\lambda},\hat{k}_{\mu} \rangle |\\
		&= |\langle A\hat{k}_{\lambda},\hat{k}_{\mu} \rangle |^2+|\langle B\hat{k}_{\lambda},\hat{k}_{\mu} \rangle |^2+2|\langle A\hat{k}_{\lambda},\hat{k}_{\mu} \rangle \langle \hat{k}_{\mu},B\hat{k}_{\lambda} \rangle |\\
		&\leq |\langle A\hat{k}_{\lambda},\hat{k}_{\mu} \rangle |^2+|\langle B\hat{k}_{\lambda},\hat{k}_{\mu} \rangle |^2+\|A\hat{k}_{\lambda}\| \|B\hat{k}_{\lambda}\|+ |\langle A\hat{k}_{\lambda},B\hat{k}_{\lambda} \rangle |,\\
		&&\,\,\,\,\,\,\,\,\,\,\,\,\,\,\,\,\,\,\,\,\,\,\,\,\,\,\,\,\,\,\,\,\,\,\,\,\,\,\,\,\,\,\,\,\,\,\,\,\,\,\,\,\,\,\,\,\,\,\,\,\,\,\,\,\,\,\,\,\,\,\,\,\,\,\,\,\,\,\,\,\,\,\,\,\,\,\,\,\,\,\,\,\,\,\,\,\,\,\,\,\,\,\,\,\,\,\,\,\,\,\,\,\,\,\,\,\,\,\,\,\,\,\,\,\,\,\,\,\,\,\,\,~\mbox{By Lemma \ref{lem3}}\\
		&\leq |\langle A\hat{k}_{\lambda},\hat{k}_{\mu} \rangle |^2+|\langle B\hat{k}_{\lambda},\hat{k}_{\mu} \rangle |^2+\frac{1}{2}(\|A\hat{k}_{\lambda}\|^2+ \|B\hat{k}_{\lambda}\|^2)+ |\langle A\hat{k}_{\lambda},B\hat{k}_{\lambda} \rangle |\\
		&\leq |\langle A\hat{k}_{\lambda},\hat{k}_{\mu} \rangle |^2+|\langle B\hat{k}_{\lambda},\hat{k}_{\mu} \rangle |^2 +\frac{1}{2}\langle (A^*A+B^*B)\hat{k}_{\lambda},\hat{k}_{\lambda} \rangle + |\langle A^*B\hat{k}_{\lambda},\hat{k}_{\lambda} \rangle |\\
		&\leq \|A\|^2_{ber}+\|B\|^2_{ber}+\frac{1}{2}\textbf{ber}(A^*A+B^*B)+\textbf{ber}(A^*B).
	\end{align*}
 Taking supremum over all $\lambda,\mu\in\Omega$, we get,
\begin{eqnarray*}
	\|A+B\|^2_{ber} \leq\|A\|^2_{ber}+\|B\|^2_{ber}+\frac{1}{2}\textbf{ber}(A^*A+B^*B)+\textbf{ber}(A^*B).
\end{eqnarray*}
Replacing $A$ by $A^*$ and $B$ by $B^*$ in the above inequality we get, 
\begin{eqnarray*}
	\|A+B\|^2_{ber} \leq \|A\|^2_{ber}+\|B\|^2_{ber}+\frac{1}{2}\textbf{ber}(AA^*+BB^*)+\textbf{ber}(AB^*).
\end{eqnarray*} 
This completes the proof. 
\end{proof}

\bibliographystyle{amsplain}

\end{document}